\documentclass[12pt,a4paper]{amsart}

\usepackage{amsmath,amsfonts,amssymb,amsthm,mathrsfs}

\usepackage{geometry}
\newcommand{\R}{\mathbb{R}}

\newcommand{\un}{\mathbf{1}\!\!{\rm I}} 
\newcommand{\be}{\begin{equation}} 
\newcommand{\ee}{\end{equation}}
\newcommand{\bea}{\begin{eqnarray}} 
\newcommand{\eea}{\end{eqnarray}}
\newcommand{\bean}{\begin{eqnarray*}} 
\newcommand{\eean}{\end{eqnarray*}}
\newcommand{\rf}[1]{(\ref {#1})}

\def\dx{\,{\rm d}x}
\def\dy{\,{\rm d}y}

\def\ds{\,{\rm d}s}

\def\eps{\varepsilon}

\def\s{\sigma}
\def\p{\partial}

\def\xn{|\!|\!|}

\def\mn{|\!\!|}

\def\mn2{|\!\!|_{M^{d/2}}}

\newtheorem{theorem}{Theorem}
\newtheorem{proposition}[theorem]{Proposition}
\newtheorem{lemma}[theorem]{Lemma}
\newtheorem{corollary}[theorem]{Corollary}
\theoremstyle{definition}

\theoremstyle{remark}
\newtheorem{remark}[theorem]{Remark}

\numberwithin{equation}{section}
\numberwithin{theorem}{section}

\author[P. Biler]{Piotr Biler}
\address{\small Instytut Matematyczny, Uniwersytet Wroc\l awski,
 pl. Grunwaldzki 2/4, \hbox{50-384} Wroc\-\l aw, Poland}
\email{Piotr.Biler@math.uni.wroc.pl}

\author[G. Karch]{Grzegorz Karch}
\address{\small 
 Instytut Matematyczny, Uniwersytet Wroc\l awski,
 pl. Grunwaldzki 2/4, \hbox{50-384} Wroc\-\l aw, Poland}
\email{Grzegorz.Karch@math.uni.wroc.pl}

\author[D. Pilarczyk]{Dominika Pilarczyk}
\address{\small Wydzia{\l} Matematyki, Politechnika Wroc{\l}awska,
Wybrze\. ze Wyspia\'nskiego 37, Wroc{\l}aw, Poland} 
\email{Dominika.Pilarczyk@pwr.edu.pl}

\title[Solutions of  chemotaxis model]{Global radial solutions in classical \\
Keller--Segel model of chemotaxis} 

\begin{document}

\begin{abstract} 
We consider the simplest parabolic-elliptic model of  chemotaxis in the whole space in several dimensions. 
Criteria   for the existence of radial global-in-time solutions  in terms of suitable Morrey norms are derived. 
\end{abstract}



\date{\today}

\thanks{  
The  first named author  was partially supported by the NCN grant 
2016/23/B/ST1/00434.  \\
The authors thank Noriko Mizoguchi  and  Ignacio Guerra  for interesting conversations.} 

\maketitle

\baselineskip=20pt

\section{Parabolic-elliptic Keller-Segel model}

\subsection*{Statement of the problem.}
We consider in this paper the following Cauchy problem in space dimensions  $d\ge 3$
\begin{align}
u_t-\Delta u+\nabla\cdot(u\nabla v)&= 0,\ \ &x\in {\mathbb R}^d,\ t>0,\label{equ}\\ 
\Delta v+u &=  0,\ \  & x\in {\mathbb R}^d,\ t>0,\label{eqv}\\
u(x,0)&= u_0(x),\ \ &x\in {\mathbb R}^d.\label{ini}
\end{align}
One motivation to study this model comes from Mathematical Biology, where equations \rf{equ}--\rf{eqv} are a simplified Keller-Segel  
system modeling chemotaxis, see e.g. \cite{B-AMSA,BCKSV,Lem}.   
The unknown variables $u=u(x,t)$ and $v=v(x,t)$ denote the density of the population of microorganisms (e.g. swimming bacteria or slime mold),  and the density of the chemical secreted by themselves that attracts them and makes them to aggregate, respectively. 

Another important interpretation of system \rf{equ}--\rf{eqv} comes from Astrophysics, where the unknown function $u=u(x,t)$ is the density of gravitationally interacting massive particles in a cloud (of stars, nebulae, etc.), and $v=v(x,t)$ is the Newtonian potential (``mean field'') of the mass distribution $u$, see \cite{Cha,CSR,B-SM,B-AM,B-CM,BHN}. 

In this work we supplement system \rf{equ}--\rf{eqv} with a nonnegative radial initial datum $u_0\ge 0$. 
Thus, by the uniqueness the corresponding solution $u$ of problem \rf{equ}--\rf{eqv} is also nonnegative and radial as long as it exists. 
It is well known that if the total mass $M$ is finite, it is conserved during the evolution 
\be 
M=\int_{\R^d}u_0(x)\dx=\int_{\R^d} u(x,t)\dx\ \ \ {\rm for\ all\ \ \ }t\in[0,T_{\rm max}).\label{M}
\ee  
 Further, we will also consider solutions with infinite mass  like the famous Chandrasekhar steady state singular solution in \cite{Cha}
\be
u_C(x)=\frac{2(d-2)}{|x|^2}.\label{Ch}
\ee 
This one-point singular solution plays a pivotal role in this work, because it allows to distinguish (in some sense) radial initial conditions of global-in-time solutions from initial data, where solutions cannot exist for all $t>0$. 
Moreover, some of the results in this paper might be  interpreted as  diffusion-dominated dynamics strictly below $u_C$, see next section.  

\subsection*{The $8\pi$-problem in the two-dimensional case} 
Let us now describe previous results which motivated us to start this study and  we  limit ourselves to those publications, which are directly related.  
We begin with the  classical case of $d=2$ where mass $M=8\pi$ plays a crucial role.    
Namely, if $u_0$ is a nonnegative measure of mass $M<8\pi$, then there exists a unique solution which is global-in-time, see {e.g.}  \cite{BM,BDP,BZ}.  
These results have been known previously for radially symmetric initial data, see  \cite{BKLN1,BKLN2,B-BCP,BKZ} for recent presentations. 
On the other hand, if $M>8\pi$, then this solution cannot be continued to a~global-in-time regular one, and a finite time blowup occurs, {cf.} \cite{B-CM,N1,K-O}, and \cite{BHN,B-BCP} for radially symmetric case. The radial blowup is accompanied by the concentration of mass  equal to $8\pi$ at the origin, \cite{BDP,B-BCP}. 
We have generalized this result in our recent work \cite{BKZ} by showing concentration of $8\pi$ for large mass solutions with sufficiently many symmetries. 

\subsection*{Parabolic-elliptic model in higher dimensions}
Now, we discuss the case   $d\geq 3$ in the Keller-Segel model. 
It is well-known that problem \rf{equ}--\rf{ini}  has a unique local-in-time mild solution $u\in {\mathcal C}([0,T); L^p(\R^d))$ 
for every $u_0\in L^p(\R^d)$ with $p>d/2$, see \cite{B-SM,K-JMAA,KS-IMUJ}. For solvability results in other functional spaces like weak Lebesgue (Marcinkiewicz), Morrey and Besov spaces, see also \cite{BB-SM,CPZ,I,K-JMAA,Lem}. In particular,   previous works have dealt with the existence of global-in-time solutions with small data in critical  spaces like $L^{d/2}(\R^d)$, $L^{d/2}_{\rm w}(\R^d)$, $M^{d/2}(\R^d)$, { i.e.}~those which are scale-invariant under the natural scaling,  see {e.g.} \cite{B-SM,BB-SM,K-JMAA,Lem}
\begin{equation}\label{scal:0}
u_\lambda(x,t)=\lambda^2 u(\lambda x,\lambda^2 t)\ \ {\rm for\ each\ \ }\lambda>0. 
\end{equation}
 Blowing up solutions to the parabolic-elliptic model of chemotaxis in dimension $d\geq 3$  have been studied in e.g. \cite{B-CM,BHN,B-AMSA,GMS,MS1,MS2,BKZ-NHM}. 
 
 \subsection*{The Keller-Segel model with an anomalous diffusion} 
 Recently, the Keller-Segel model has been considered in a more general form 
 \begin{align}
u_t+(-\Delta)^{\alpha/2} u+\nabla\cdot(u\nabla v)&= 0,\ \ &x\in {\mathbb R}^d,\ t>0,\label{equ-a}\\ 
\Delta v+u &=  0,\ \  & x\in {\mathbb R}^d,\ t>0,\label{eqv-a} 
\end{align}
with motivations stemming still from biology. 
 Here, the diffusion process  is given by the  fractional power of the Laplacian $(-\Delta)^{\alpha/2}$ with  $\alpha\in(0,2)$, a nonlocal L\'evy type operator,  as was in \cite{BCKZ,BKZ2} and \cite[Sec. 10]{Lem}.  
Results in \cite{BKZ2}, with their proofs  based on new comparison principles for nonlocal operators,  are in a sense parallel to these on global as well as blowing up solutions in the present work. 
In particular, the singular solution $u_C(x)={s(\alpha,d)}{|x|^{-\alpha}}$ in \cite[(2.1)]{BKZ2},  generalizing the Chandrasekhar solution \rf{Ch}, determines a kind of threshold between global-in-time and blowing up solutions.

\subsection*{Brief summary of results in this work.}

In view of the above mentioned results for $d=2$,  a natural question can be posed in dimension  $d\ge 3$: 
{\it Does there exist a critical quantity $\ell=\ell(u_0)$ which decides whether a solution with $u_0$ as the initial datum is global-in-time or this blows up in a finite time?} 
More precisely:  
Do there exist constants $0<c(d)\le C(d)<\infty$ such that the assumption $\ell(u_0)<c(d)$ implies global-in-time existence of solution to problem \rf{equ}--\rf{eqv} while the condition $\ell(u_0)>C(d)$ leads to a finite time blowup of solution? 

We will give a partial  answer to that question for radially symmetric solutions 
 showing that  the quantity $\ell=\ell(u_0)$ corresponds  to the  radial concentration defined in  \rf{rconc} below --- and thus it is equivalent to the Morrey norm in the space  $M^{d/2}(\R^d)$.
Similar results   in the case $\alpha\in(1,2)$  are in \cite{BKZ2,BZ-2} and in  a forthcoming book  \cite{B-book}.


Our main results include:  
global-in-time existence of radially symmetric solutions with  initial data in the critical  Morrey space $M^{d/2}(\R^d)$ whose initial conditions are uniformly below the singular solution $u_C$ in \rf{Ch} in an averaged sense in  Theorem \ref{glo}, together with 
convergence of solutions to $0$ as $t\to\infty$ in various norms. 
The proof of the first result involves a pointwise argument, a~powerful tool used in different contexts such as  free boundary problems and fluid dynamics, cf. also \cite{B-bl} for the case of a nonlinear heat equation. 
A sufficient condition for the global-in-time existence is, in fact, an estimate of the Morrey space  $M^{d/2}(\R^d)$ norm of the initial condition.

As a preliminary result  on general (including sign-changing and not necessarily radial) solutions we will show the 
local-in-time existence of solutions with initial data in Morrey spaces $M^{d/2}(\R^d)\cap M^p(\R^d)$, $p\in\left(\frac{d}{2},d\right)$ in Proposition \ref{lok-istn}. 
The proof  of this auxiliary result is  based on the classical Fujita--Kato iterations procedure in suitably chosen spaces contained in the critical space $M^{d/2}(\R^d)$ which admit local singularities only in $M^p(\R^d)$ thus weaker than the critical one of $u_C$.

We recall in Appendix \ref{blo} some blowup criteria,  in  \ref{ss} and  \ref{sss} ---  some properties of nonnegative stationary radial solutions and nonnegative selfsimilar radial solutions. 


\subsection*{Notation.}  
The $L^p(\R^d)$ norm is denoted by $\|\, .\, \|_p$, $1\le p\le\infty$. 
We use systematically the 
 homogeneous Morrey spaces $M^p(\R^d)$ 
 of measurable functions  on  $\R^d$ which are defined by their norms 
\be
|\!\!| u|\!\!|_{M^p}\equiv  \sup_{R>0,\, x\in\R^d}R^{d(1/p-1)} 
\int_{\{|y-x|<R\}}|u(y)|\dy<\infty\label{hMor}
\ee  
with $M^1(\R^d)=L^1(\R^d)$ and $M^\infty(\R^d)=L^\infty(\R^d)$. 
In the following, we shall also use 
the {\it radial concentration} of a function $u=u(x)$ given by 
\be
\xn u\xn
\equiv  \sup_{R>0}R^{2-d}\int_{\{|y|<R\}}u(y)\dy \label{rconc}
\ee
which, in the case of radial functions, appears to be equivalent to the Morrey norm $M^{d/2}(\R^d)$, see \cite[Lemma 7.1]{BKZ2} as well as \cite[Lemma 3.1]{ADB}. 

 \section{Statement of the results} \label{ls} 
Equation \rf{eqv} is not uniquely solved with respect to $v$.
Thus, in this work we always assume that 
\be
\nabla v=\nabla E_d\ast u  \label{fundsol}
\ee
with  
$$
 E_2(x)=\tfrac{1}{2\pi}\log\tfrac{1}{|x|}\qquad  {\rm and}\qquad  E_d(x)=\tfrac{1}{(d-2)\s_d}\tfrac{1}{|x|^{d-2}}\ \ {\rm for}\ \ d\ge 3. 
$$
Here, the number 
 \be
\s_d=\frac{2\pi^{\frac{d}{2}}}{\Gamma\left(\frac{d}{2}\right)} \label{pole}
\ee 
is the measure of the unit sphere ${\mathbb S}^{d-1}\subset \R^d$. 
Consequently, we consider in fact the Cauchy problem for the nonlocal transport equation 
\begin{align}\label{KS}
u_t-\Delta u+\nabla\cdot(u\nabla E_d\ast u)&= 0,  &&x\in\R^d,\ ,\\
u(x,0)&=u_0(x), &&x\in\R^d.\nonumber
\end{align}
Let us first formulate the main result of this section on the global-in-time solutions to the nonlocal problem \rf{KS}.

\begin{theorem}[Global-in-time solutions]\label{glo}
Let $d\geq 3$. 
Assume that a radially symmetric nonnegative initial condition $u_0\in M^{d/2}(\R^d)$ satisfies  
\be\label{a1}
\xn u_0\xn\equiv\sup_{R>0}R^{2-d}\int_{\{|x|<R\}} u_0(x)\dx<2\s_d.
\ee 
There exists $p\in \left(\frac{d}{2},d\right)$ such that if moreover $u_0\in M^p(\R^d)$ then the corresponding solution $u=u(x,t)$ of problem \rf{KS} exists in the space
\begin{equation*}
  {\mathcal C}_w \Big( [0,T), M^{d/2}(\R^d) \cap M^p(\R^d) \Big)\cap \Big\{ u: (0,T) \to L^\infty (\R^d): t^{\frac{d}{2p}}\| u(t)\|_\infty <\infty \Big\}
\end{equation*}
for each $T>0$,  is nonnegative, global-in-time, radial, and satisfies the bound 
\begin{equation}\label{global_sol}
\xn u(t)\xn=\sup_{R>0}R^{2-d}\int_{\{|x|<R\}} u(x,t)\dx< 2\s_d\ \ {\rm for\ all}\ \ t>0.
\end{equation}
\end{theorem}

Theorem \ref{glo} is proved in the following way. 
First, in Proposition \ref{lok-istn} we construct local-in-time solutions for arbitrary initial conditions (even  sign changing and not necessarily radial)  from $M^{d/2}(\R^d)\cap M^p(\R^d)$. 
Such solutions for radial nonnegative data appear to be also radial, nonnegative and sufficiently regular. 
Then, we derive a differential equation for the radial distribution function $M(r,t)=\int_{\{|x|<r\}}u(x,t)\dx$ and study its properties in Proposition \ref{maximum}. 
Theorem \ref{glo}  is proved at the end of Section \ref{sec:apriori}.

\begin{remark}
Notice that, by a direct calculation, we have 
$$
2\s_d=R^{2-d}\int_{\{|x|<R\}}u_C(x)\dx\ \ \ {\rm for\ each}\ \ \ R>0.
$$
Hence, the assumption $\xn u_0\xn<2\s_d$ in Theorem \ref{glo} means that the initial datum $u_0$ is \underline{strictly} below  $u_C$ in the sense of the radial concentrations, namely, we assume that for some $\varepsilon \in (0,1)$ we have 
$$
R^{2-d}\int_{\{|x|<R\}}u_0(x)\dx<\varepsilon R^{2-d}\int_{\{|x|<R\}}u_C(x)\dx \ \ \ {\rm for\ each}\ \ \ R>0.
$$
By Theorem \ref{glo}, the corresponding solution exists for all $t\ge 0$ and stays below the singular steady state $u_C$ in this sense. 
\end{remark}

\begin{remark}
In fact, we have identified 
that the radial concentration \eqref{rconc}
 decides whether a nonnegative, radially symmetric solution to problem \rf{equ}--\rf{ini} exists globally in time or it blows up in a finite time. More precisely, we have dichotomy
\begin{itemize}
   \item if $\xn u_0 \xn <2\s_d$, then a solution is global-in-time by Theorem \ref{glo}, on the other hand
   \item there exists a constant $C_d>0$ such that if $\xn u_0\xn >C_d$ then a solution cannot be global-in-time, see Appendix \ref{blo} below for corresponding blowup results.
\end{itemize}
\end{remark}

Obviously, the constant in the blowup criterion satisfies $C_d>2\s_d$. Below, we estimate the discrepancy between $C_d$ and $2\s_d$.

\begin{remark}\label{ess}
For each $T>0$, system \eqref{equ}--\eqref{eqv} has an explicit  
{\em smooth}   
blowing up solution which satisfies
\be
\int_{|x|\le r}u(x,t)\dx=\frac{4\s_d r^d}{r^2+2(d-2)(T-t)}
\qquad \text{for all}\quad t\in [0,T),
\label{bl-sol}
\ee
 see \cite[(33)]{BCKSV} with a discussion in \cite{BZ-2}.
Here, the corresponding initial density  is
$$
u_0(x)=4(d-2)\frac{|x|^2+2T}{(|x|^2+2(d-2)T)^2}$$
and it  satisfies 
$$ \xn u_0\xn=4\s_d=\lim_{r\to\infty}r^{2-d}\int_{|x|\le r}u(x,t)\dx=\xn u(t)\xn
\qquad \text{for all}\quad t\in [0,T).
$$
Notice that  $u_0 \in M^{d/2}(\R^d) \cap L^\infty (\R^d) \subset M^{d/2}(\R^d) \cap M^p(\R^d)$ for $p\in \left(\frac{d}{2}, d\right)$. Since there exists a blowing up solution with the radial concentration equal to $4\sigma_d$ then, necessarily, we should have $C_d\geq 4\sigma_d$ in the blowup criterion in Theorem~\ref{blow}.
Papers \cite{BKZ,BKZ-NHM,BZ-2} contain  estimates of this  constant $C_d$ from above. Actually, the best bound $C_d<4\s_d\sqrt{\pi d} $ has been obtained in \cite{BZ-2}. 
\end{remark}

In the case of an integrable initial datum, the global-in-time solution constructed in Theorem  \ref{glo} decays in the following sense.

\begin{corollary}[Decay of radial concentration]\label{asy-z}
Under  the assumptions of Theorem  \ref{glo}, if additionally the initial condition   satisfies $u_0\in L^1(\R^d)$,  then 
\be
 \xn u(t)\xn=\sup_{R>0}R^{2-d}\int_{\{|x|<R\}}u(x,t)\dx\to 0\ \ \ {\rm as}\ \ \ t\to\infty. 
 \label{as-z}
\ee 
\end{corollary}

\begin{corollary}[$L^q$-decay estimate] \label{cor:Lp:decay}
Under the assumptions of Theorem  \ref{glo}, if  the initial condition also satisfies $u_0\in L^1(\R^d)$,  then
 $$
 \|u(t)\|_q\leq C(d,q,\mu) t^{-\frac{d}{2}\left(1-\frac{1}{q}\right)}\|u_0\|_1 \quad \text{for all}
 \quad  t>0.
 $$
\end{corollary}

\begin{remark}\label{stationary}
For $d\ge 3$ problem \rf{equ}--\rf{eqv} has stationary solutions $U=U(x)$ and $\nabla V = \nabla E_d \ast U$ which satisfy
$$
\lim_{R\to\infty}R^{2-d}\int_{\{|x|\le R\}}U(x)\dx = 2\s_d,
$$
see Theorem \ref{oscillating} below.
They do not belong to $L^1(\R^d)$ and, of course, their radial concentrations are constant in time. 
\end{remark}

\begin{remark}\label{selfsim}
For the initial condition $u_0(x) = \varepsilon |x|^{-2}$ with sufficiently small $\varepsilon >0$, problem \rf{equ}--\rf{ini} (or \rf{KS}) has selfsimilar solutions of the form 
$$u(x,t)=\frac{1}{t}U\left(\frac{|x|}{\sqrt{t}}\right)$$ 
with a profile $U$, see Appendix \ref{sss}. Here, the radial concentration is also (as in Remark~\ref{stationary}) constant in time, because
\begin{equation*}
\begin{split}
\sup_{R>0}R^{2-d}\int_{\{|x|\le R\}}
\frac{1}{t}U\left(\frac{|x|}{\sqrt{t}}\right)\,{\rm d}x
=\sup_{R>0} \left(\frac{R}{\sqrt{t}}\right)^{2-d}
\int_{\left\{|x|\le \frac{R}{\sqrt{t}}\right\}}
U(|x|)\,{\rm d
}x=\xn U\xn .
\end{split}
\end{equation*}
The quantity $\xn U\xn$ is finite because $U \in {\mathcal C}^\infty (\R^d)$ and $|U(\varrho)|\le \frac{C}{|\varrho|^2}$ for large $\varrho$, see Appendix~\ref{sss}. Notice that the initial datum of such a selfsimilar solution satisfies  $u_0 \in M^{d/2}(\R^d)\setminus M^p(\R^d)$ for each $p>\frac{d}{2}$ and $\xn u_0\xn = \varepsilon\frac{\s_d}{d-2}$.
\end{remark}

\section{Local-in-time solutions}
We begin by constructing a local-in-time {\em mild} solution of problem \rf{KS}, namely, a~function $u=u(x,t)$ satisfying the Duhamel formula  
\be
u(t)={\rm e}^{t\Delta}u_0+  {\mathcal B}(u,u)(t).\label{Duh}
\ee 
Here, the symbol ${\rm e}^{t\Delta}$  denotes the heat  semigroup  on $\R^d$, and the bilinear form ${\mathcal B}$ is defined by 
\be
{\mathcal B}(u,w)(t)=\int_0^t \nabla {\rm e}^{(t-s)\Delta}(u\nabla E_d\ast w)(s)\ds.\label{B-form}
\ee 

To deal with the integral equation \rf{Duh}, we recall that,  
similarly to the action of the heat semigroup on $L^p$ spaces, for $1\le p\le q \le \infty$ 
the inequalities 
\begin{align}
|\!\!| {\rm e}^{t\Delta}f|\!\!|_{M^{q}}&\le Ct^{-d(1/p-1/q)/2}|\!\!|f|\!\!|_{M^p} \label{polgrupa},\\
|\!\!| {\nabla \rm e}^{t\Delta}f|\!\!|_{M^{q}}&\le Ct^{-1/2-d(1/p-1/q)/2}|\!\!|f|\!\!|_{M^p} \label{polgrupa1}
\end{align}
hold, where for either $p=1$ or $q=\infty$ the above estimate involve the norms $\|\, .\,\|_1$ and $\|\, .\, \|_\infty$ norms, resp., (see \cite[Prop. 3.2]{G-M}). 
We also recall from \cite[Prop. 3.1]{G-M} a version of  Riesz potential estimates in the Morrey norms  
\begin{equation}\label{Sob0}
|\!\!| \nabla E_d* u|\!\!|_{M^r} \leq C|\!\!| u|\!\!|_{M^p}\quad \text{with} \quad 
\frac{1}{r} =\frac{1}{p}-\frac{1}{d}
\end{equation}
as well as an interpolation estimate 
\be
\| \nabla E_d* u\|_\infty\le C|\!\!| u|\!\!|_{M^{p}}^\mu|\!\!| u|\!\!|_{M^r}^\nu\label{Sob}
\ee
with 
$$1\le p<d<r\le\infty,\ \ {\rm and} \ \ \mu=\frac{\frac1d-\frac1r}{\frac1p-\frac1r},\ \  \nu=\frac{\frac1p-\frac1d}{\frac1p-\frac1r},\ \ {\rm  so\ that}\ \ \mu+\nu=1.$$ 

\begin{proposition} \label{lok-istn}
Given $u_0\in M^{d/2}(\R^d)\cap M^p(\R^d)$ with $d\ge 2$ and $p\in\left(\frac{d}{2},d\right)$, there exist $T=T(u_0)>0$ and a unique local-in-time solution 
\begin{equation}\label{X_T_space}
u\in {\mathcal X}_T = {\mathcal C}_{\rm w}\Big([0,T], M^{d/2}(\R^d)\cap M^p(\R^d)\Big)
\cap 
{\mathcal Y}_T,
\end{equation}
where
$$
\mathcal{Y}_T = \Big\{u:(0,T)\to L^\infty(\R^d): \ \ 
\sup_{0<t\le T}t^{\frac{d}{2p}}\| u(t)\|_{\infty}<\infty\Big\}
$$
of problem \rf{equ}--\rf{ini}. This is a classical solution of this problem, namely, $u, u_t, \nabla u, D^2 u \in {\mathcal C}\left(\R^d \times (0,T)\right)$.
\end{proposition}

\begin{remark}\label{weak-conv}
Note that, in general, we have only  weak convergence of ${\rm e}^{t\Delta}u_0$ to an initial datum $u_0\in M^p(\R^d)$ with $1< p <\infty$. Thus, we are obliged to consider weakly continuous ${\mathcal C}_{\rm w}\left([0,T], M^{d/2}(\R^d)\cap M^p(\R^d)\right)$ instead of norm continuous functions of time variable $t$ with values in a Morrey space. 
\end{remark} 

\begin{remark}\label{assumptions} 
Note that $u_C\in M^{d/2}(\R^d)$ and $|\!\!|u_C|\!\!|_{M^{d/2}}=\xn u_C\xn=2\s_d$. 
The second assumption $u_0\in M^p(\R^d)$, with   $p>\frac{d}{2}$, is a kind of regularity assumption which rules out local singularities of  strength $\frac{1}{|x|^2}$. 
Indeed, $\un_{\{|x|<R\}}u_C\not\in M^p(\R^d)$ while $\un_{\{|x|>R\}}u_C\in M^p(\R^d)$ with $p>\frac{d}{2}$.  
\end{remark}

\begin{remark}
In fact, if $u_0\in M^{d/2}(\R^d)$ is sufficiently small then the solution in Proposition~\ref{lok-istn} exists for all $t\in [0,\infty)$, see \cite[Thm. 1 B), C)]{Lem} with a proof involving homogeneous Morrey spaces of two indices. 
\end{remark}

\begin{proof}[Proof of Proposition \ref{lok-istn}.] 
We sketch the proof of this result noting that the reasonings are, in a sense, close to those in \cite[Prop. 1, Th. 1]{B-SM}. For $T>0$ and $p\in \left(\frac{d}{2}, d\right)$ consider an auxiliary space ${\mathcal X}_T$ defined in \rf{X_T_space}. We supplement the space ${\mathcal X}_T$ with the usual norm 
\begin{equation*}
  \| u\|_{{\mathcal X}_T} =\sup_{0\le t\le T} |\!\!| u(t) |\!\!|_{M^{d/2}}+ \sup_{0\le t\le T}|\!\!| u(t)|\!\!|_{M^p} + \sup_{0<t\le T}t^{\frac{d}{2p}}\| u(t)\|_{\infty}.
\end{equation*}
By estimates \rf{polgrupa},
we get ${\rm e}^{t\Delta}u_0\in{\mathcal X}_T$ whenever 
$u_0\in M^{d/2}(\R^d) \cap M^p(\R^d)$.  More precisely, we have
\be
\left\|{\rm e}^{t\Delta}u_0\right\|_{{\mathcal X}_T}\le C\Big(|\!\!| u_0|\!\!|_{M^{d/2}}+|\!\!| u_0|\!\!|_{M^p}\Big)\label{Xest}
\ee 
holds with a constant $C$ {\em independent} of $T$. 
To find solutions to equation \eqref{Duh}, it suffices  to show that the bilinear form ${\mathcal B}$ in \rf{B-form} is continuous in the norm of the ${\mathcal X}_T$. 

Let $u, w \in {\mathcal X}_T$ and $t \in [0, T]$. For $q\in\left\{\frac{d}{2},p\right\}$ we set $\nu = 1-\frac{q}{d} \in (0,1)$ and proceed in the following way using systematically inequalities \rf{polgrupa}--\rf{Sob}  
\begin{align*}
|\!\!| {\mathcal B}(u,w)(t)|\!\!|_{M^q}&\le C\int_0^t(t-s)^{-\frac{1}{2}}|\!\!| u(s)\nabla E_d* w (s)|\!\!|_{M^q}\ds\\
&\le C\int_0^t(t-s)^{-\frac{1}{2}}|\!\!| u(s)|\!\!|_{M^q}\|\nabla E_d * w(s)\|_\infty\ds\\
&\le C\int_0^t(t-s)^{-\frac{1}{2}}|\!\!| u(s)|\!\!|_{M^q}|\!\!| w(s)|\!\!|_{M^q}^{1-\nu}\| w(t)\|_\infty^{\nu}\ds\\
&\le C\int_0^t(t-s)^{-\frac{1}{2}}s^{-\nu\frac{d}{2p}} |\!\!| u(s)|\!\!|_{M^q}|\!\!| w(s)|\!\!|_{M^q}^{1-\nu}\left(s^{\frac{d}{2p}}\| w(t)\|_\infty\right)^\nu\\
&\le CT^{\frac{1}{2}-\nu\frac{d}{2p}}\times \sup_{0\le s \le T}|\!\!| u(s)|\!\!|_{M^q}\times \sup_{0\le s\le T}|\!\!| w(s)|\!\!|_{M^q}^{1-\nu}\times \|w\|_{{\mathcal Y}_T}^\nu .
\end{align*}
Here, for  $q=d/2$ we have $\frac{1}{2}-\nu\frac{d}{2q}=0$, while for $q=p>d/2$ we obtain $\frac{1}{2}-\nu\frac{d}{2p}>0$. 

We deal with the last part of the norm $\| \cdot \|_{\mathcal{X}_T}$ analogously using inequalities \rf{polgrupa} and \rf{polgrupa1} with $r\in (d, \infty)$, $p\in \left(\frac{d}{2},d\right)$ satisfying $\frac{1}{r}=\frac{1}{p}-\frac{1}{d}$ in the following way
\begin{align}\label{L_infty}
t^\frac{d}{2p}\| {\mathcal B}(u,&w)(t)\|_\infty \le  t^\frac{d}{2p} C\int_0^t (t-s)^{-\frac{1}{2}-\frac{d}{2r}}\| u(s)\|_\infty |\!\!|\nabla E_d* w (s)|\!\!|_{M^r} \ds\\  \nonumber
&\le Ct^\frac{d}{2p} \int_0^t (t-s)^{-\frac{1}{2}-\frac{d}{2r}}s^{-\frac{d}{2p}}\left(s^{\frac{d}{2p}}\| u(s)\|_\infty\right) |\!\!| w(s)|\!\!|_{M^p}\ds\\ \nonumber
&\le CT^{1-\frac{d}{2p}}\| u\|_{{\mathcal Y}_T}\times \sup_{0\le s\le T}|\!\!| w(s)|\!\!|_{M^p}\nonumber
\end{align}
because 
\begin{equation*}
\frac{d}{2p}-\frac{1}{2}-\frac{d}{2r}-\frac{d}{2p}+1=1-\frac{d}{2p}>0.
\end{equation*}
Now, a construction of local-in-time solutions is completed by the Banach fixed point argument used in the same way as stated, e.g., in \cite[Thm. 1]{Lem}, \cite[Lemma 5.1]{BCKZ}, \cite[Proposition 1]{B-SM}.

This local-in-time solutions satisfies $u\in L^\infty \big((\delta, T), L^\infty (\R^d)\big)$ for each $\delta >0$. Moreover, by inequality \rf{Sob0} with $p\in\left(\frac{d}{2},d\right)$ and $r=\infty$ we have got $\nabla E_d * u \in L^\infty \big((\delta, T), L^\infty (\R^d)\big)$. Hence, a standard regularity argument for parabolic equations permits us to prove smoothness of $u = u(x,t)$.
\end{proof}

\section{A priori estimates and continuation of solutions}\label{sec:apriori}

Now, we are going to obtain {\em a~priori} estimates of a local-in-time solution in Proposition~\ref{lok-istn} using its radial distribution function 
\be
 M(r,t)=\int_{\{|x|<r\}}u(x,t)\dx\label{rdf}
\ee
which satisfies the equation 
\be\label{mass}
\frac{\partial M}{\partial t}=M_{rr}-\frac{d-1}{r}M_r+\frac{1}{\sigma_d}r^{1-d}MM_r,
\ee 
with $M(0,t)=0$, cf. e.g. \cite{BHN}. 
Recall that   the radial  function $u=u(x)$ is related to $M(r)=\int_{\{|x|<r\}}u(x)\dx $ by the equality
\be
u(x)=\frac{1}{\sigma_d}r^{1-d}\frac{\partial}{\partial r}M(r),\label{u-rad}
\ee
 satisfied for each $|x|=r$. 

The estimate of the radial concentration $\xn u(t)\xn$ in Theorem \ref{glo} is a consequence of the following property of the function $M(r,t)$.

\begin{proposition}\label{maximum}
Let $d\ge 3$ and $p\in \left(\frac{d}{2}, d\right)$. Consider the nonnegative, sufficiently smooth, radial, local-in-time solution $u=u(x,t)$ on $[0,T]$ corresponding to a nonnegative radial initial datum $u_0 \in M^{d/2}(\R^d) \cap M^p(\R^d)$ (see Proposition \ref{lok-istn}). 
Suppose that
\be
M_0(r)<\min\left\{ K\, r^{d-d/p},\eps 2\s_d\, r^{d-2}\right\}\equiv b(r)\ \ \textit{for all}\quad r>0,\label{M0}
\ee
for some $K>0$ and $\eps \in \left(0, \frac{d}{2p}\right)$ where $\frac{d}{2p}<1$. Then the solution $M$ of equation \rf{mass} with the initial condition $M_0$ satisfies the same bound 
\be
M(r,t) = \int_{|x|<r} u(x,t) \dx  <\min\left\{ K\, r^{d-d/p},\eps 2\s_d\, r^{d-2} \right\}\ \ \textit{for}\quad r>0\ \ \textit{and}\quad t >0.\label{Mt}
\ee 
\end{proposition}

\begin{proof} 
The proof of inequality \rf{Mt} is by contradiction. 
Suppose that the function $M(r,t)$ hits the barrier $b(r)$ defined in \rf{M0} at some $(R_0,t_0)$ with $R_0>0$ and $t_0>0$ chosen as the least such moment $t>0$. The two parts of the graph of the barrier $b(r)$ meet at $r_\ast=\left(\frac{\eps 2\s_d}{K}\right)^{\frac{1}{2-d/p}}$. If $R_0\ge r_\ast$, then $z(r,t)=r^{2-d}M(r,t)$ hits the constant level $2\eps \s_d$ at $r=R_0$ and $t_0$. On the other hand, if $0<R_0< r_\ast$, then the function $\tilde z(r,t)=r^{d/p-d}M(r,t)$ hits the constant level $K$ at $r=R_0$ and $t_0$.

By a simple argument (see e.g. \cite[Thm. 2.4]{BKZ-NHM}), one may show that the function $z(r,t_0)$ attains its local maximum at $R_0$ where 
\be \label{cond-max}
\frac{\p}{\p t}z(R_0,t_0)\ge 0, \quad z_r(R_0,t_0)=0,\quad z_{rr}(R_0,t_0)\le 0.
\ee
These relations hold true for $\tilde z$, as well. 

Now, we use the equation for the function $z(r,t)=r^{2-d}M(r,t)$ which is obtained immediately from \rf{mass} and takes the form
\be 
(r^{d-2}z)_t=(r^{d-2}z)_{rr}-\frac{d-1}{r}(r^{d-2}z)_r+\frac{1}{\sigma_d}r^{1-d}(r^{d-2}z)(r^{d-2}z)_r. \label{zz}
\ee
Applying relations \rf{cond-max} and the property of $(R_0,t_0)$ we have got
\begin{align*}
\frac{\p}{\p t}z(R_0,t_0)&= z_{rr}(R_0,t_0)-2\frac{d-2}{R_0^2}z(R_0,t_0)+\frac{d-2}{\sigma_dR_0^2}z^2(R_0,t_0) \\
&\le \frac{d-2}{\sigma_dR_0^2}z(R_0,t_0)(z(R_0,t_0)-2\sigma_d)\\
&=\frac{4(d-2)}{R_0^2}\s_d \eps (\eps-1) <0
\end{align*}
because $\eps <1$. This is a contradiction with the first relation in \rf{cond-max}. 

In the case of the function $\tilde z$ (notice that $r^{d-d/p}\tilde z=M$), we consider the analogous equation 
\be
(r^{d-d/p}\tilde z)_t=(r^{d-d/p}\tilde z)_{rr}-\frac{d-1}{r}(r^{d-d/p}\tilde z)_r+\frac{1}{\sigma_d}r^{1-d}(r^{d-d/p}\tilde z)(r^{d-d/p}\tilde z)_r, 
\label{zzz}
\ee 
also equivalent to equation \rf{mass}. We use this equation at the point $(R_0,t_0)$ and apply relations \rf{cond-max} with $z$ replaced by $\tilde z$ to obtain
\begin{align*}
\frac{\p}{\p t}\tilde z(R_0,t_0)&= \tilde z_{rr}(R_0,t_0)-\tfrac{d/p(d-d/p)}{R^2}\tilde z(R_0,t_0) +\tfrac{d-d/p}{\s_d}R_0^{-d/p}\tilde z^2(R_0,t_0)\\
&\le \tfrac{d-d/p}{\s_d R_0^2}\tilde z(R_0,t_0)\left(-\tfrac{d}{p}\s_d + R_0^{2-d/p}\tilde z(R_0,t_0)\right)\\
&=\tfrac{d-d/p}{\s_d R_0^2}K \left(-\tfrac{d}{p}\s_d + r_\ast^{2-d/p}K\right)\\
&=\tfrac{d-d/p}{R_0^2}K \left(-\tfrac{d}{p} + 2\eps \right) <0
\end{align*}
since $R_0<r_\ast = \left(\frac{\eps 2\s_d}{K}\right)^{\frac{1}{2-d/p}}$ and $2\eps<\frac{d}{p}$. This is again a contradiction with first relation in \rf{cond-max}. 
 \end{proof}
 
We are in a position to complete the proof of our main result.

\begin{proof}[Proof of Theorem \ref{glo}]
Let us first recall that in the case of each nonnegative radial function $v$ following inequalities hold true
\begin{equation}\label{v_Morey}
  C |\!\!| v|\!\!|_{M_{p}} \le \sup_{R>0}R^{d(1/p-1)}\int_{\{|x|<R\}}v(x) \dx \le |\!\!|v|\!\!|_{M^{p}}
\end{equation}
with $p\in (1,d]$ and a number $C=C(p)\in (0,1)$, see \cite[Lemma 7.1]{BKZ2}. 

Let $u$ be a local-in-time solution on $[0,T]$ constructed in Proposition \ref{lok-istn} corresponding to a nonnegative radial initial datum $u_0\in M^{d/2}(\R^d)\cap M^p (\R^d)$ with $p\in \left( \frac{d}{2},d\right)$. Obviously, the function $u=u(x,t)$ is nonnegative and radial. In order to continue this solution globally in time, it suffices to show an \textit{a priori} estimate $\|u\|_{{\mathcal X}_T} <\infty$ for each $T>0$. 

First, using inequality \rf{v_Morey} we have got
\begin{equation*}
  M_0(r) \le r^{d-d/p}|\!\!|u_0|\!\!|_{M^p} <Kr^{d-d/p} \quad \text{for all} \quad r>0
\end{equation*}
with an arbitrary number $K>|\!\!|u_0|\!\!|_{M^p}$. Because of assumption \rf{a1} there exists $\eps \in (0,1)$ such that 
\begin{equation*}
  \int_{\{|x|<r\}} u_0(x) \dx <2\eps \s_dr^{d-2} \quad \text{for all}\quad r>0.
\end{equation*}

Now, we choose $p\in \left(\frac{d}{2},d\right)$ so close to $\frac{d}{2}$ in order to have $\eps<\frac{d}{2p}$. Thus, by Proposition~\ref{maximum} and the first inequality in \rf{v_Morey} we obtain
\begin{equation*}
  \sup_{0\le t<T}|\!\!| u(t) |\!\!|_{M^{d/2}} + \sup_{0\le t<T}|\!\!| u(t) |\!\!|_{M^p} <\infty .
\end{equation*}

Let us show an analogous estimate for the $L^\infty$-norm. Computing the $L^\infty$-norm of equation \rf{Duh} and following estimates in \rf{L_infty} with $p\in \left(\frac{d}{2},d\right)$ and $w=u$ we have
\begin{align*}
  \|u(t)\|_{\infty} \le Ct^{-\frac{d}{2p}}|\!\!| u_0|\!\!|_{M^p}+C\sup_{0\le s<T}|\!\!| u(s)|\!\!|_{M^p}\int_0^t (t-s)^{-\frac{1}{2}-\frac{d}{2p}}\| u (s)\|_\infty \ds.
\end{align*}
Hence, by the singular Gronwall lemma \cite[Lemma 1.2.9]{CD} we obtain
\begin{equation*}
  \sup_{0<t<T} t^{\frac{d}{2p}}\| u(t)\|_\infty <\infty .
\end{equation*}

By a standard continuation argument, the solution can be continued globally in time. The estimate \rf{global_sol} is an immediate consequence of Proposition \ref{maximum}.
\end{proof}

\section{Decay estimates} 
 Let us now show that our global-in-time finite mass solutions  decay in time.

\begin{proof}[Proof of Corollary \ref{asy-z}] 
It follows from Theorem \ref{glo} (in fact, from Proposition \ref{maximum}) that
\begin{equation}\label{M_est}
  M(r,t) = \int_{|x|<r} u(x,t) \dx <2\s_d r^{d-2} \quad \text{and}\quad M(r,t) \le Kr^{d-d/p}.
\end{equation}
Moreover, this is a bounded function because of the conservation of mass
\begin{equation}\label{M_0}
  M(r,t) \le \int_{\R^d} u(x,t) \dx = \int_{\R^d} u_0(x)\dx \equiv M_0\quad \text{for all}\quad t\ge 0.
\end{equation}
Obviously, we have $M_r(r,t) \ge 0$ for nonnegative $u(x,t)$. Thus, using equation \rf{mass} and the first inequality in \rf{M_est} we obtain 
\begin{align}
\frac{\p M}{\p t}&\le M_{rr}-\frac{d-1}{r}M_r+\frac{2\s_d}{2\s_d r}M_r\nonumber\\
&= M_{rr}-\frac{d-3}{r}M_r\le M_{rr}\label{heat}
\end{align}
since $d\ge 3$. Using the comparison principle and assumptions of Proposition \ref{maximum} we obtain that $0\le M(r,t) \le m(r,t)$, where $m=m(r,t)$ is a solution of the following initial-boundary value problem for the heat equation on the half line
\begin{align}\label{1C}
  m_t& =m_{rr}, &\quad &  (r,t)\in (0,\infty)\times(0,T), \nonumber\\
 m(0,t)&= 0 &\quad & {\rm for } \quad t\ge 0,\\
 m(r,0)&= m_0(r) &\quad & {\rm for }\quad r>0\nonumber\\
\intertext{with}
0\le m_0(r)&\le \min\left\{Kr^{d-d/p},2\s_d r^{d-2}, M_0\right\}&\quad & {\rm for }\quad r> 0\nonumber.
\end{align}
As it is well known, the solution of the initial-boundary value problem \rf{1C} can be represented by the formula
\begin{equation}\label{1heat} 
m(r,t)= \frac{1}{(4\pi t)^{1/2}}\int_0^{\infty }\left({\rm e}^{-(r-x)^2/(4t)}-{\rm e}^{-(r+x)^2/(4t)}\right)m_0(x)\dx.
\end{equation}
Our goal is to show that $\sup_{r>0}r^{2-d}m(r,t) \to 0$ as $t\to \infty$. To do it, let us fix $0<R_1<R_2<\infty$ with small $R_1$ and large $R_2$ to be determined later. 
By the second inequality in \rf{M_est} we have 
\begin{equation}\label{z1}
   r^{2-d}m(r,t) \le K\, r^{2-d/p}\le KR_1^{2-d/p}\quad \text{for} \quad r\in(0,R_1)
\end{equation}
since $\frac{d}{p}\in(1,2)$. Next, by inequality \rf{M_0}, we get 
\begin{equation}\label{z2}
  r^{2-d}m(r,t) \le M_0 r^{2-d}\le M_0 R_2^{2-d}\quad \text{for}\quad r\in(R_2,\infty),
\end{equation}
since $d\ge 3$. Finally, using equation \rf{1heat} we obtain 
\begin{align*}
 m(r,t) &\le  M_0 \frac{1}{(4\pi t)^{1/2}}\int_0^\infty\left({\rm e}^{-(r-y)^2/(4t)}-{\rm e}^{-(r+y)^2/(4t)}\right)\dy\\
&=  M_0 \frac{1}{(4\pi t)^{1/2}}\int_{-r}^r{\rm e}^{-\rho^2/(4t)}\,{\rm d}\rho\\
&= M_0 \frac{1}{(4\pi)^{1/2}}\int_{-\frac{r}{\sqrt{t}}}^{\frac{r}{\sqrt{t}}}{\rm e}^{-\varrho^2/4}\,{\rm d}\varrho
\end{align*}
which imply 
\begin{align}\label{z3}
\sup_{r\in[R_1,R_2]} r^{2-d}m(r,t) \le M_0R_1^{2-d}\frac{1}{(4\pi)^{1/2}}\frac{2R_2}{\sqrt{t}}. 
\end{align}
Putting together inequalities \rf{z1}, \rf{z2}, \rf{z3}, we arrive at 
\begin{equation*}
\limsup_{t\to\infty}\left( \sup_{r>0} r^{2-d}m(r,t)\right) \le KR_1^{2-d/p}+ M_0 R_2^{2-d}
\end{equation*}
and the right hand side can be done as small as we wish with a suitable choice of small $R_1$ and large $R_2$.

Therefore, $\lim_{t\to\infty}\sup_{r>0}r^{2-d}M(r,t) =0$ holds for every nonnegative, radial initial datum satisfying $u_0\in M^{d/2}(\R^d)\cap M^p(\R^d)\cap L^1(\R^d)$ and $\xn u_0\xn<2\s_d$. 
\end{proof}

To obtain the $L^p$-decay of solutions stated in  Corollary \ref{cor:Lp:decay}, we need a lemma which is an immediate consequence of the Gauss theorem.

\begin{lemma}[{\cite[Lemma 2.1]{BKZ} and  \cite{BZ-2}}]\label{potential}
Let $u\in L^1_{\rm loc}(\R^d)$ be a radially symmetric function,  
 such that  $v=E_d\ast u$ with $E_2(x)=-\frac1{2\pi}\log|x|$ and $E_d(x)=\frac1{(d-2)\sigma_d}|x|^{2-d}$ for $d\ge 3$, solves the Poisson equation $\Delta v+u=0$. Then the identity 
$$
\nabla v(x)\cdot x=-\frac1{\sigma_d}|x|^{2-d}\int_{\{|y|\le |x|\}} u(y)\dy
$$
holds. 
\end{lemma}

\begin{proof}[Proof of Corollary \ref{cor:Lp:decay}.] 
Let $u=u(x,t)$ be a nonnegative, radial and sufficiently smooth solution of problem \rf{KS}. We multiply equation \rf{equ} by $u^{q-1}$ with some $q>1$ and after integrations by parts we obtain 
\begin{align*}
\frac1q\frac{\rm d}{{\rm d}t}\int u^q\dx &=
-(q-1)\int|\nabla u|^2u^{q-2}\dx +(q-1)\int u\nabla v\cdot\nabla u\, u^{q-2}\dx \\
&=-4\frac{q-1}{q^2}\int|\nabla u^{q/2}|^2\dx +2\frac{q-1}{q}\int u^{q/2}\nabla v\cdot\nabla u^{q/2}\dx\\
&\le -4\frac{q-1}{q^2}\int|\nabla u^{q/2}|^2\dx +2\frac{q-1}{q}\int \frac{u^{q/2}}{|x|}|x\cdot \nabla v||\nabla u^{q/2}|\dx.
\end{align*}
In the last line, we have use the inequality   $ |\nabla v |\leq |x|^{-1}|x\cdot \nabla v|$ valid for the radial function $v=v(x)$. 
Using Lemma \ref{potential} we have
\begin{equation*}
|x\cdot \nabla v(x,t)|=\frac{1}{\s_d}r^{2-d}M(r,t).
\end{equation*} 
Hence, recalling the Hardy inequality (see \cite{BEL}) 
\be
\frac{(d-2)^2}{4}\left\|\frac{f}{|x|}\right\|_2^2\le \|\nabla f\|_2^2\label{Ha}
\ee
we arrive at the estimate 
\begin{equation}\label{q_estimate}
\frac{\rm d}{{\rm d}t}\|u\|^q_q \le -4\frac{q-1}{q}\|\nabla u^{q/2}\|_2^2 +2(q-1)\frac{2}{q-2}\frac{1}{\s_d}\sup_{r>0}r^{2-d}M(r,t)\|\nabla u^{q/2}\|_2^2
\end{equation}
for sufficiently large $t>0$. Since, by Corollary \ref{asy-z}, we have got $\sup_{r>0} r^{2-d}M(r,t)\to 0$ as $t\to \infty$, we obtain
\begin{equation}\label{mu_estimate}
\frac{\rm d}{{\rm d}t}\|u\|^q_q+\mu \|\nabla u^{q/2}\|_2^2\le 0
\quad \text{with a constant}\quad
\mu>0
\end{equation}
for sufficiently large $t>0$. Using the Gagliardo-Nirenberg inequality
\begin{equation*}
  \|w\|_2^{2+s}\le C(d,p)\|\nabla w\|_2^2\|w\|_p^{s}  \quad \text{for}\quad 1<p<2\quad \text{and} \quad C(d,p)>0,
\end{equation*}
where $s=\frac{4p}{d(2-p)}$ we obtain
\begin{align*}
  \frac{\rm d}{{\rm d}t}\|u\|^q_q \le -\mu \|\nabla u^{q/2}\|_2^2 \le -C(d,p) \frac{\|u^{q/2}\|_2^{2+s}}{\|u^{q/2}\|_p^s}.
\end{align*}
We take here $p=\frac{2}{q}$ to get
\begin{align*}
  \frac{\rm d}{{\rm d}t}\|u\|^q_q \le -C(d,p) \frac{\|u\|_q^{1+s/2}}{\|u_0\|_1^{qs/2}}.
\end{align*}
 with $s=\frac{4}{d(q-1)}$ and $C(d,q, \mu)>0$, 
since the mass conservation property for nonnegative solutions gives $\|u(t)\|_1= \|u_0\|_1$.  
This leads to the differential inequality of the form
\begin{equation*}
  f'(t) \le -C \|u_0\|_1^{-qs/2}f^{1+s/2}(t)
\end{equation*} 
for the function $f(t)=\|u(t)\|^q_q$ and $C>0$, which finally gives an algebraic decay of the $L^q$-norm
\begin{equation*}
\|u(t)\|_q \le Ct^{-d/2(1-1/q)}\|u_0\|_1 \quad \text{for sufficiently large}\quad t>0.
\end{equation*}
\end{proof}

\begin{remark}[$L^q$-estimate] \label{cor:Lp}
The $L^q$-decay estimates from Corollary \ref{cor:Lp:decay} is obtained from the corresponding $L^q$-energy inequality \rf{mu_estimate} which can be also proved for nonintegrable initial data in higher dimensions. Indeed, coming back to estimate \rf{q_estimate} and using the estimate $r^{2-d}M(r,t)<2\s_d$ from Theorem \ref{glo} we obtain
\begin{equation*}
\frac{\rm d}{{\rm d}t}\|u\|^q_q+\mu \|\nabla u^{q/2} \|_2^2\le 0
\qquad
\text{with}\qquad  \mu=4(q-1)\left(\tfrac1q-\tfrac{2}{d-2}\right)>0
\end{equation*}
for $d\ge 5$ and  $1< q<\frac{d}{2}-1$.
\end{remark}

\bigskip

\appendix

\section{Complements on blowing up and special solutions}

Here, we discuss  some of  known results on the Keller-Segel model  \rf{equ}--\rf{eqv} which are related to Theorem \ref{glo}.

\subsection{Blowup of solutions}\label{blo}

We recall here sufficient conditions for nonexistence of solutions.
  \begin{theorem}\cite[Sec. 8]{BKZ2}\label{blow}
Consider a local-in-time, nonnegative,   classical, radially symmetric solution $u\in {\mathcal C}([0,T), L^1_{\rm loc}(\R^d))$ of   problem \rf{equ}--\rf{eqv} with a nonnegative radially symmetric initial datum $u_0\in  L^1_{\rm loc}(\R^d)$. 

\begin{itemize} 
\item[\textit{(i)}] There exists a constant $C_{d}>0$ such that if 
\be
\sup_{R>0}R^{2-d}\int_{\{|x|<R\}} u_0(x)\dx>C_{d},\label{(i)}
\ee 
then the solution $u$ cannot exists for all $t>0$.  
 \item[\textit{(ii)}]   If, moreover,  
\be
\limsup_{R\to 0}R^{2-d}\int_{\{|x|<R\}} u_0(x)\dx>C_{d},
\label{(ii)}
\ee then the solution $u(x,t)$ cannot be defined  on any time interval $[0,T]$ with some $T>0$.
\end{itemize} 
\end{theorem}  

This Theorem is proved in \cite{BKZ2} in the case of system \rf{equ-a}--\rf{eqv-a} with $\alpha\in(0,2]$. The proof involves  moments 
$$
w_R(t)=\int \psi_R(x)u(x,t)\dx\label{moment}
$$
defined with a continuous bump function $\psi$ and its rescalings for $R>0$ 
\be
\psi(x)=(1-|x|^2)_+^{1+\frac{\alpha}{2}}
=\left\{
\begin{array}{ccc}
(1-|x|^2)^{1+\frac{\alpha}{2}}& \text{for}& |x|<1,\\
0& \text{for}& |x|\geq 1,
\end{array}
\right.
\qquad  \psi_R(x)=\psi\bigg(\frac{x}{R}\bigg). \label{bump}
\ee 
We refer the reader to \cite[Proposition 2.6]{BZ-2}, for a proof that the following  qualitative sufficient conditions for blowup for the radial initial datum $u_0\ge 0$ in case $\alpha=2$ 
\begin{itemize}
\item $ \sup_{t>0}t {\rm e}^{t\Delta}u_0(0)>2$,

\item $\sup_{t>0}t\left\| {\rm e}^{t\Delta}u_0\right\|_\infty\gg 1$, 

\item $\xn u_0\xn\equiv \sup_{r>0}r^{2-d}\int_{\{|x|<r\}} u_0(x)\dx \gg 1$,

\item $|\!\!|u_0\mn2 \equiv \sup_{r>0, \,x\in\R^d}r^{2-d}\int_{\{|y-x|<r\}} u_0(y)\dy \gg1$, 
\end{itemize}
are mutually equivalent.


\subsection{Stationary solutions}\label{ss}

Let us now recall an  old result on stationary solutions of the Keller-Segel system.
\begin{theorem}\cite[Sec. 6]{B-AMSA}\label{oscillating}
Let $d\ge 3$. 
There exist radially symmetric stationary solutions $U=U(x)$, $V=V(x)$ with $\nabla V=\nabla E_d\ast U$ satisfying the relation 
$$
\lim_{R\to\infty}R^{2-d}\int_{|x|\le R}U(x)\dx = 2\s_d.
$$
If $3\le d \le 9$, the radial concentrations $r^{2-d}\int_{|x|\le r}U(x)\dx$ of these solutions oscillate in $r$ infinitely many times around the value $2\s_d$. 
\end{theorem}

We sketch below the proof, first introducing tools applied in \cite{B-AM,B-AMSA}. Time-independent solutions of equation \rf{mass} satisfy the following equation 
\be
M_{rr}-\frac{d-1}{r}M_r+\frac{1}{\sigma_d}r^{1-d}MM_r=0.\label{stats}
\ee
If $d\ge 3$ and $M\not\equiv 0$ then $\lim_{r\to\infty}M(r)=\infty$;  more precisely
\be 
\lim_{r\to\infty}r^{2-d}M(r)=2\s_d.\label{as-ss}
\ee
 Moreover, $ M(r)\le 2(d-1)2\s_dr^{d-2}$ which is obtained by multiplying equation \rf{stats} by $r^{d-1}$ and integrating on $[0,r]$. The change of variables $\tau=\log r$, $\dot{\ }=\frac{\rm d}{{\rm d}\tau}$ and 
$$
X(\tau)=\tfrac{1}{\s_d}r^{3-d}M_r(r),\ \ \ Z(\tau)=\tfrac{1}{\s_d}r^{2-d}M(r)$$ 
replaces the second order equation \rf{stats} with a dynamical system in the plane 
\bea
\dot X&=&(2-Z)X,\label{X}\\
\dot Z&=&X-(d-2)Z.\label{Z}
\eea
This system has two stationary solutions $(X,Z)=(0,0)$ and $(X,Z)=(2(d-2),2)$ and possesses  a Lyapunov function 
\be
 L(X,Z)=\tfrac12(Z-2)^2+\left(X-2(d-2)-2(d-2)\log\tfrac{X}{2(d-2)}\right)\label{Lya}
\ee
such that $\frac{\rm d}{{\rm d}\tau}L=-(d-2)(Z-2)^2$.  
The quadrant $\{X>0,\ Z>0\}$ is invariant for system \rf{X}--\rf{Z} so that 
each solution of  equation \rf{stats} with $M(0)\ge 0$, $M_r(0)\ge 0$ is positive and nondecreasing. 
We are interested in the, so-called, {\em eternal} solutions defined on the whole real line $\R\ni\tau$.

The linearization of system \rf{X}--\rf{Z} at the point $(2(d-2),2)$ has two  complex conjugate eigenvalues with negative real part if $3\le d\le 9$, and two negative real eigenvalues if $d\ge 11$.  
Hence,  there exists a separatrix joining the unstable point $(0,0)$ (when $\tau\to-\infty$) with the  stable point $(2(d-2),2)$ (when $\tau\to\infty$).  Its slope at the origin is equal to $1/d$. 
Observe that if $3 \le d \le 9$ then $r^{2-d}M(r)$  is not monotone, since the separatrix turns around the point $(2(d-2), 2)$ infinitely many times. 
For $d = 2$ and $d \ge 10$, there is a~unique scroll of such a curve, since the
eigenvalues of the linearization of the vector field at $(2(d-2), 2)$ are real.

It can be shown that the separatrix joining the stationary points $(0, 0)$ and $(2(d-2), 2)$ is the unique trajectory with $X(\tau) > 0$ for all $\tau\in \R$.  
and this satisfies  $X(-\infty)=0$, $Z(-\infty)=0$. 
Therefore, all the stationary solutions are parametrized by, say, their values $Z(\tau)$  at $\tau=0$ corresponding to $\frac{1}{\s_d}M(r)$ at $r=1$. This determines the number of such stationary solutions with a given value $M(1)$. 
For $3\le d\le 9$ there might be multiple stationary solutions with the same value of $M(1)$. 
All of them satisfy the relation $\lim_{\tau\to\infty}Z(\tau)=2$, so that relation \rf{as-ss} follows.


\subsection{Selfsimilar solutions}\label{sss}
The positive radially symmetric selfsimilar solutions, emanating from  initial data multiples of order $\frac{1}{|x|^2}$, do not decay as $t\to\infty$ in the sense of their radial concentration. 
Indeed, they are of the form 
$$u(x,t)=\frac{1}{t}U\left(\frac{|x|}{\sqrt{t}}\right)$$ 
with a profile $U$ and $u_0(x)=\varepsilon u_C$ with some (small) $\varepsilon>0$, see Remark \ref{selfsim}. 
 Thus, they do not satisfy assumptions of Theorem \ref{glo}. 
 According to  to the result in \cite[Theorem 1 B), C)]{Lem} mentioned in the Introduction (see also \cite[Theorem 3]{B-SM}, \cite[Theorem 4]{B-AMSA}) they are small in the space $M^{d/2}(\R^d)$ 
 so that the corresponding functions $z$ and their profiles $\zeta$   keep small values with respect to $r$ (w.r.t. $y$, resp.) below  $\eps2\s_d$,  at least for very small $0<\eps\ll 1$.  
Here we gather some facts on the selfsimilar solutions whose proofs can be found in \cite{B-SM,B-AM}. 

Solutions invariant under the scaling leaving system \rf{equ}--\rf{eqv} unchanged are of the form $u(x,t)=\frac1tU\left(\frac{x}{\sqrt{t}}\right)$ and their initial data are homogeneous of degree $-2$. Thus, radial positive selfsimilar solutions which are not stationary have as their initial values $u_0=\eps u_C$ with small $\eps>0$, anyway $\eps<1$. 

For  them,  the integrated density is of the form $M(r,t)=\s_dt^{d/2-1}\zeta\left(\frac{r^2}{t}\right)$, and  the profile $\zeta=\zeta(y)$, $y=\frac{r^2}{t}$, $'=\frac{\rm d}{{\rm d}y}$, satisfies 
\be
\zeta''+\tfrac14\zeta'-\tfrac{d-2}{2y}\zeta'-\tfrac{d-2}{8y}\zeta+\tfrac{1}{2y^{d/2}}\zeta\zeta'=0,\ \ \ \zeta(0)=0,\ \ \ \zeta(y)\approx 2\eps y^{d/2-1},\ y\to\infty.\label{zeta}
\ee
It can be shown that each solution of problem \rf{zeta} verifies estimates 
\bea
\zeta(y)&\le& \left(1-\tfrac{2}{d}\right)y^{d/2}+4(d-1)y^{d/2-1},\label{es1}\\
\quad&& \lim_{y\to\infty}y^{1-d/2}\zeta(y)\in(0,\infty).\label{es2}
\eea
The  change of variables $\tau=\frac12\log y$, $\dot{\ }=\frac{\rm d}{{\rm d}\tau}$, 
$$
X(\tau)=2y^{2-d/2}\zeta'(y),\ \ \ Z(\tau)=y^{1-d/2}\zeta(y),
$$ 
similar to that in Appendix \ref{ss},  leads to a nonautomous system in the plane 
\bea
\dot X&=&(2-Z)X+\tfrac{{\rm e}^{2\tau}}{2}((d-2)Z-X), \label{Xs}\\
\dot Z&=&X-(d-2)Z. \label{Zs}
\eea
Selfsimilar solutions emanate from the origin $(0,0)$ and terminate at the points 
\newline $(2(d-2)\eps,2\eps)$, $\eps\ll 1$, moving in the sector $\{ X>(d-2)Z\}$ since $\dot Z(\tau)>0$. 




\begin{thebibliography}{99}

\bibitem{ADB}
D. Andreucci, E. DiBenedetto, 
{\it On the Cauchy problem and initial traces for a class of 
evolution equations with strongly nonlinear sources}, 
Ann. Sc. Norm. Super. Pisa, Cl. Sci., IV. Ser. {\bf 18} (1991), 363--441. 

\bibitem{BEL}
A. A. Balinsky, W. D. Evans, R. T. Lewis, 
{\it The Analysis and Geometry of Hardy's Inequality}, Universitext, Springer, Cham (2015).

\bibitem{BM} 
J. Bedrossian, N. Masmoudi, 
{\it Existence, uniqueness and Lipschitz dependence for Patlak-Keller-Segel and Navier-Stokes in $\mathbb R^2$ with measure-valued initial data},  
Arch. Rational Mech. Anal. {\bf 214} (2014), 717--801. 
  
\bibitem{B-SM} 
P. Biler, {\it The Cauchy problem and self-similar solutions for a nonlinear 
parabolic equation,} Studia Math. {\bf 114} (1995),  181--205. 

\bibitem{B-CM}
P. Biler, {\it Existence and nonexistence of solutions for a model
of gravitational interaction of particles\ \ III},  Coll. Math. {\bf 68} (1995), 229--239.

\bibitem{B-AM}
P. Biler, {\it Growth and accretion of mass in an astrophysical model}, Applicationes Math. (Warsaw) {\bf 23} (1995), 179--189. 

\bibitem{B-AMSA}
P. Biler, {\it Local and global solvability of parabolic systems modelling chemotaxis}, Adv. Math. Sci. Appl. {\bf 8} (1998), 715--743. 

 \bibitem{B-BCP}  
 P. Biler, 
 {\it   Radially symmetric solutions of a chemotaxis model in the plane -- the supercritical case}, 31--42, 
 in: {\it Parabolic and Navier-Stokes Equations}, Banach Center Publications {\bf 81}, Polish Acad. Sci.,  Warsaw, 2008. 
 
\bibitem{B-bl} 
P. Biler, 
{\it  Blowup  versus global in time existence of  solutions for nonlinear heat equations}, 1--13. Topol. Meth. Nonlin. Analysis, to appear;  arXiv:1705.03931v2.  

\bibitem{B-book}
P. Biler, 
{\it Singularities of Solutions to Chemotaxis Systems}, book in preparation,  
 DeGruyter,  Series in Mathematics and Biological Sciences. 


\bibitem{BB-SM} 
P. Biler, L. Brandolese, 
{\it On the parabolic-elliptic limit of the doubly parabolic Keller--Segel system  modelling chemotaxis},  
Studia Math. {\bf 193} (2009), 241--261.   

  \bibitem{BCKZ}
P. Biler, T. Cie\'slak, G. Karch, J. Zienkiewicz,
 {\it Local criteria for blowup of solutions  in two-dimensional chemotaxis models}, Disc.  Cont. Dynam. Syst. A {\bf 37} (2017), 1841--1856.   
 
 \bibitem{BHN}  
 P. Biler, D. Hilhorst, T. Nadzieja, 
 {\it Existence and nonexistence of solutions for a model 
of gravitational interaction of particles\ \ II,} Colloq. Math. {\bf 67} (1994), 297--308. 

 
\bibitem{BKLN1} 
P. Biler, G. Karch, Ph. Lauren\c cot, T. Nadzieja, 
{\it  The $8\pi$-problem for radially symmetric solutions of a chemotaxis model in a disc},
 Topol. Methods Nonlin. Anal. {\bf 27} (2006), 133--147.

\bibitem{BKLN2}
P. Biler, G. Karch, Ph. Lauren\c cot, T. Nadzieja, 
{\it The $8\pi$-problem for radially symmetric solutions of a chemotaxis model in the plane},
Math. Methods in the Applied Sci. {\bf 29} (2006),  1563--1583. 

\bibitem{BKZ}
 P. Biler, G. Karch, J. Zienkiewicz, 
 {\it Optimal criteria for blowup of radial and $N$-symmetric solutions of chemotaxis systems},  Nonlinearity {\bf 28} (2015), 4369--4387.  
 
  
\bibitem{BKZ-NHM} 
 P. Biler, G. Karch, J. Zienkiewicz, 
 {\it Morrey spaces norms and criteria for blowup in chemotaxis models},    Networks and NonHomogeneous Media {\bf 11} (2016), 239--250. 

\bibitem{BKZ2}
 P. Biler, G. Karch, J. Zienkiewicz, 
{\it Large global-in-time solutions to a nonlocal model of chemotaxis},  Adv. Math. {\bf 330} (2018), 834--875. 
 
\bibitem{BZ}  
  P. Biler, J. Zienkiewicz, 
{\it Existence of solutions for the Keller-Segel model of chemotaxis with measures as initial data}, 
Bull.  Polish Acad.  Sci. Mathematics {\bf 63} (2015), 41--52.  

 \bibitem{BZ-2}
P. Biler,   J. Zienkiewicz,  
{\it Blowing up radial solutions in the minimal Keller-Segel chemotaxis model}, submitted. 

\bibitem{BCKSV}
M. P. Brenner, P. Constantin, L. P. Kadanoff, A. Schenkel, S. C. Venkataramani, {\it Diffusion, attraction and collapse}, Nonlinearity {\bf 12} (1999), 1071--1098. 

\bibitem{BDP} 
A. Blanchet, J. Dolbeault, B. Perthame, 
{\it Two-dimensional {K}eller-{S}egel  model: optimal critical mass and qualitative properties of the solutions}, 
Electron. J. Differential Equations \textbf{44}, 32 pp. (2006).    
 
\bibitem{Cha}
S. Chandrasekhar, 
{\it Principles of Stellar Dynamics}, University of Chicago Press, Chicago (1942).

\bibitem{CSR}
P. H. Chavanis, J. Sommeria, R. Robert, 
{\it Statistical mechanics of two-dimensional vortices and and collisionless stellar systems}, The Astrophys. Journal {\bf 471}  (1996), 385--399. 

\bibitem{CD}
J.W. Cholewa T. Dlotko, {\it Global Attractors in Abstract Parabolic Problems}, London Mathematical
Society Lecture Notes Series, Vol. {\bf 278}, Cambridge University Press, Cambridge, 2000.

\bibitem{CPZ}
L. Corrias, B. Perthame, H. Zaag, 
{\it Global solutions of some chemotaxis and angiogenesis system in high space dimension}, Milan J. Math. {\bf 72} (2004), 1--28. 
 
 \bibitem{G-M}
 Y. Giga, T. Miyakawa, 
 {\it Navier-Stokes flow in ${\mathbb R}^d$ with measures as initial vorticity and Morrey spaces}, 
 Commun. Partial Differ. Equations {\bf 14}  (1989), 577--618. 

\bibitem{GMS}
Y. Giga, N. Mizoguchi, T. Senba, 
{\it Asymptotic behavior of type I blowup solutions to a parabolic-elliptic system of drift-diffusion type}, Arch. Rational Mech. Anal. {\bf 201} (2011), 549--573.  

\bibitem{I}
T. Iwabuchi, 
{\it Global well-posedness for Keller-Segel system in Besov type spaces}, 
J. Math. Anal. Appl. {\bf 379} (2011), 930--948. 

\bibitem{K-JMAA}
G. Karch, 
{\it  Scaling in nonlinear parabolic equations},  
J. Math. Anal. Appl. {\bf 234}  (1999), 534--558.  

\bibitem{KS-IMUJ} 
H. Kozono, Y. Sugiyama, 
{\it The Keller-Segel system of parabolic-parabolic type with initial data in weak $L^{n/2}(\mathbb R^n)$ and its application to self-similar solutions},
 Indiana Univ. Math. J. {\bf 57} (2008), 1467--1500.

  \bibitem{K-O} 
 M. Kurokiba, T. Ogawa, 
 {\it Finite time blow-up of the solution for a nonlinear parabolic equation of drift-diffusion type}, 
 Differ. Integral Equ. {\bf 16} (2003), 427--452. 
 

\bibitem{Lem} 
P.-G. Lemari\'e-Rieusset, 
{\it Small data in an optimal Banach space for the parabolic-parabolic and parabolic-elliptic Keller-Segel equations in the whole space}, Adv. Diff. Eq. {\bf 18} (2013), 1189--1208. 


\bibitem{MS1} 
N. Mizoguchi, T. Senba, \textit{A sufficient condition for type I blowup in a parabolic-elliptic system}, J.~Differential Eq. \textbf{250} (2011), 182--203. 

\bibitem{MS2} 
N. Mizoguchi, T. Senba, \textit{Type-II blowup of solutions to an elliptic-parabolic system}, Adv. Math. Sci. Appl. \textbf{17} (2007), 505--545. 
 


\bibitem{N1} 
T. Nagai, 
{\it Blowup of nonradial solutions to parabolic-elliptic systems modeling chemotaxis in two-dimensional domains}, J. Ineq. Appl. {\bf 6} (2001),  37--55.


\end{thebibliography}
   \end{document}